\documentclass[12pt,a4paper,final]{amsproc}

\usepackage{amssymb}
\usepackage{amscd}
\usepackage[all]{xy}

\usepackage[cp850]{inputenc}
\usepackage{mathrsfs}
\usepackage[notcite,notref]{showkeys}
\begin{document}
\title[Regarding Kalton]{An example regarding Kalton's paper ``Isomorphisms between spaces of vector-valued continuous functions''           }
\author{F\'{e}lix Cabello S\'{a}nchez}
\thanks{Research supported in part by MICIN Project PID2019-103961GB-C21.}


\address{Departamento
de Matem\'{a}ticas and IMUEx, Universidad de Extremadura}
\address{Avenida de Elvas,
06071 Badajoz, Spain.} \email{fcabello@unex.es}
\date{}
\noindent {\footnotesize Version \today}

\bigskip

\bigskip

 \maketitle
\theoremstyle{plain}
\newtheorem{problem}{Problem}
\newtheorem{theorem}{Theorem}
\newtheorem{proposition}{Proposition}
\newtheorem*{corollary}{Corollary}
\newtheorem*{lemma}{Lemma}
\newtheorem*{question}{Question}
\theoremstyle{remark}
\newtheorem{definition}{Definition}
\newtheorem{example}{Example}
\newtheorem{examples}{Examples}
\newtheorem{remark}{Remark}
\newcommand{\R}{\mathbb{R}}
\newcommand{\N}{\mathbb{N}}
\newcommand{\K}{\mathbb{K}}
\newcommand{\e}{\varepsilon}


\newcommand{\LH}{\mathcal L(\mathscr H)}
\newcommand{\Loo}{\ensuremath{{L_\infty }}}

\newcommand{\EL}{\ensuremath{\Ext_{L^\infty}}}

\newcommand{\To}{\ensuremath{\longrightarrow}}

\def\FS{\operatorname{FVLS}}
\def\FV{\operatorname{FVLV}}
\def\IV{\operatorname{IV}}
\def\Fp{\operatorname{F-$p$-B}}
\def\Hom{\operatorname{Hom}}
\def\PB{\operatorname{PB}}
\def\PO{\operatorname{PO}}
\def\El{\operatorname{Ext}}
\def\sign{\operatorname{sign}}
\def\dim{\operatorname{dim}}
\def\diam{\operatorname{diam}}
\def\supp{\operatorname{supp}}
\def\dens{\operatorname{dens}}
\def\co{\operatorname{co}}
\def\Bil{\operatorname{\mathfrak B}}
\def\Bal{\operatorname{Bal}}
\def\dist{\operatorname{dist}}
\def\M{\operatorname{M}}
\bibliographystyle{plain}



The paper alluded to in the title contains the following striking result \cite[Theorem 6.4]{k-pems}: 
Let $I$ be the unit interval and $\Delta$ the Cantor set. If $X$ is
a quasi Banach space containing no copy of $c_0$ which is isomorphic to
a closed subspace of a space with a basis and $C(I,X)$ is linearly homeomorphic to $C(\Delta, X)$, then $ X $ is locally convex, i.e., a Banach space.

Here $C(K,X)$ denotes the space of continuous functions $F:K\To X$. When $K$ is a compact space and $X$ a quasi Banach space $C(K,X)$ is also a quasi Banach space under the quasinorm $\|F\|=\sup\{\|F(t)\|: t\in K\}$.

When $X$ is a Banach space, the isomorphic theory of the spaces $C(K,X)$ is somehow {oversimplified} by Miljutin theorem (the spaces $C(K)=C(K,\mathbb R)$ for $K$ uncountable and metrizable are all mutually isomorphic) and, above all, by Grothendieck's identity $C(K,X)=C(K)\check{\otimes}_{\varepsilon} X$
which implies that the isomorphic type of the Banach space $C(K,X)$ depends only on those of $C(K)$ and $X$. The situation for quasi Banach spaces is more thrilling and actually some seemingly innocent questions remain open: Is $C(I,\ell_p)$ isomorphic to $C(I^2,\ell_p)$? Is $C(I,L_p)$ isomorphic to $C(\Delta,L_p)$? These appear as Problems 7.2 and 7.3 at the end of \cite{k-pems}. Problem 7.1, namely if $C(K)\otimes X$ (the subspace of functions whose range is contained in some finite dimensional subspace of $X$) is always dense in $C(K,X)$, was posed by Klee and is connected with quite serious mathematics. While it seems to be widely open for quasi Banach spaces $X$ the answer is negative for $F$-spaces (complete linear metric spaces) as shown by Cauty's celebrated example \cite{cauty} (see also \cite{d-k}) and affirmative for locally convex spaces. See Waelbroeck \cite[Section~8]{w} for a discussion on Klee's {\em density} problem.

The aim of this short note is much more modest: we will show that Kalton's result is sharp by exhibiting non-locally convex quasi Banach spaces $X$ with a basis for which $C(I,X)$ and $C(\Delta, X)$ are isomorphic. Our examples are rather specific and actually in all cases $X$ is isomorphic to $C(K,X)$ if $K$ is a metric compactum of finite covering dimension.
\medskip

Recall that the (Lebesgue) covering dimension  of a (not necessarily compact) topological space $K$ is the smallest number $n\geq 0$ such that every open cover admits a refinement in which every point of $K$ lies in the intersection of no more than $n+1$ sets of the refinement.

A quasi Banach space $X$ has the $\lambda$-approximation property ($\lambda$-AP) if for every $x_1,\dots, x_n\in X$ (or in some dense subset) there is a finite-rank operator $T$ on $X$ such that $\|T\|\leq\lambda$ and $\|x_i-Tx_i\|<\e$. We say that $X$ has the bounded approximation property (BAP) if it has the  $\lambda$-AP for some $\lambda\geq 1$.

We end these preliminaires by recalling that a $p$-norm, where $0<p\leq 1$, is a quasinorm satisfying the inequality $\|x+y\|^p\leq \|x\|^p+\|y\|^p$ and that every quasinormed space has an equivalent $p$-norm for some  
$0<p\leq 1$, so says the Aoki-Rolewicz theorem.

\begin{lemma}
If $K$ has finite covering dimension and $X$ has the BAP, then $C(K,X)$ has the BAP. 
\end{lemma}

\begin{proof}
We first observe that if $K$ has finite covering dimension {\em or} $X$ has the BAP, then $C(K)\otimes X$ is dense in $C(K,X)$.
The part concerning the BAP is obvious; the other part is a result by Shuchat \cite[Theorem 1]{s-pams}.

Given $g\in C(K)$ and $x\in X$ we denote by $g\otimes x$ the function $t\longmapsto g(t)x$.
Since every function in $C(K)\otimes X$ can be written as a finite sum $\sum_{i} g_i\otimes x_i$ with $g_i\in C(K), x_i\in X$ (which justifies our notations, see \cite[Proposition~1]{s-pams})
it suffices to see that there is a constant $\Lambda$ such that, given $f_1,\dots,f_m\in C(K), y_1,\dots,y_m\in X$ and $\e>0$ there is a finite-rank operator $T$ on $C(K,X)$ such that $\|T\|\leq \Lambda$ and 
$\|f_i\otimes y_i - T(f_i\otimes y_i)\|<\e$. As $\e$ is arbitrary there is no loss of generality in assuming that $\|f_i\|=\|y_i\|=1$ for $1\leq i\leq m$.

Take an open cover $U_1,\dots,U_r$ of $K$ such that for every $i,j$ one has $|f_i(s)-f_i(t)|<\e$ for all $s,t\in U_j$.
Put $n=\dim(K)$ and 
take a refinement $V_1,\dots, V_k$ so that each point of $K$ lies in no more than $n+1$ of those sets. Finally, let $\phi_1, \dots, \phi_k$ be a partition of unity of $K$ subordinate to $V_1,\dots, V_k$.

For each $j$, pick $t_j\in V_j$ and define an operator $L$ on $C(K,X)$ by letting $L(F)=\sum_{j\leq k} \phi_j\otimes F(t_j)$, that is, 
$(LF)(t)=\sum_{j\leq k} \phi_j(t) F(t_j)$. Let us estimate $\|L\|$ assuming $X$ is $p$-normed: one has
\begin{align*}
\|L(F)\|= \sup_{t\in K}\Big\|\sum_{j\le k}  \phi_j(t) F(t_j)\Big\|,
\end{align*}
but for each $t\in K$ the sum has no more than $n+1$ nonzero summands, so
\begin{gather*}
\Big\|\sum_{j\le k}  \phi_j(t) F(t_j)\Big\| \leq \Big( \sum_{j\le k}  \phi_j(t)^p \|F(t_j)\|^p\Big)^{1/p}\leq \|F\| \Big( \sum_{j\le k}  \phi_j(t)^p \Big)^{1/p}\\
\leq \|F\|\|\,{\bf I}: \ell_1^{n+1}\To \ell_p^{n+1}\|=(n+1)^{1/p -1}\|F\|
\end{gather*}
We claim that $\|f_i\otimes y_i - L(f_i\otimes y_i)\|\le\e$ for all $i$. We have
$
L(f_i\otimes y_i)(t)=\sum_{j\leq k}f_i(t_j)\phi_j(t) y_i
$,
hence
$$\|f_i\otimes y_i - L(f_i\otimes y_i)\|= \Big\|\sum_{j\leq k} f_i\phi_j -\sum_{j\leq k} f_i(t_j)\phi_j\Big\|\|y_i\| = \Big\|\sum_{j\leq k} f_i\phi_j -\sum_{j\leq k} f_i(t_j)\phi_j\Big\|.
$$
For each $j$ and each $t\in K$ one has $|f_i(t)\phi_j(t)-f_i(t_j)\phi_j(t)|\le\e \phi_j(t)$: this is obvious if $t\notin V_j$ since in this case $\phi_j(t)=0$, while for $t\in V_j$ we have $|f_i(t)-f_i(t_j)|\le\e 
$ by our choice of $V_1,\dots, V_k$ and thus
$$
\Big|\sum_{j\leq k} f_i(t)\phi_j(t) -\sum_{j\leq k} f_i(t_j)\phi_j(t)\Big|
\le\e \sum_{j\leq k}\phi_j(t)=\e
$$
holds for all $t\in K$; consequently we have
$$
\Big\|\sum_{j\leq k} f_i\phi_j -\sum_{j\leq k} f_i(t_j)\phi_j\Big\|\leq \e.$$

Let $R$ be a finite-rank operator on $X$ such that $\|y_i-R(y_i)\|<\e$, with $\|R\|\leq\lambda$, where $\lambda$ is the ``approximation constant'' of $X$, and define $T$ on $C(K,X)$ by
$
(TF)(t)=R((LF)(t))
$.
Clearly $T$ has finite-rank since for an elementary tensor $f\otimes x$ one has
$$
T(f\otimes x)= \sum_{j\leq k} f(t_j) \phi_j\otimes R(x).
$$
Finally let us estimate $\|f_i\otimes y_i - T(f_i\otimes y_i)\|$. 
Write
$$
f_i\otimes y_i - T(f_i\otimes y_i)
= f_i\otimes y_i - f_i\otimes R(y_i) + f_i\otimes R(y_i) -
\sum_{j\leq k} f_i(t_j) \phi_j\otimes R(y_i)
$$
 and then
\begin{gather*}
\| f_i\otimes y_i - T(f_i\otimes y_i)\|^p\leq 
\| f_i\otimes y_i - f_i\otimes R(y_i)\|^p + \Big\|f_i\otimes R(y_i) -
\sum_{j\leq k} f_i(t_j) \phi_j\otimes R(y_i) \Big\|^p\\
= 
\| f_i\|^p\|y_i - R(y_i)\|^p + \Big\|f_i -
\sum_{j\leq k} f_i(t_j)  \Big\|^p \| R(y_i) \|^p \leq \e^p+ \e^p \lambda^p,
\end{gather*} 
so that $C(K,X)$ has the BAP with constant at most  $\lambda(n+1)^{1/p -1}$.
\end{proof}

The proof raises the question of whether the lemma is true for, say, the Hilbert cube $I^\omega$.

The other ingredient we need is a {\em complementably universal} space for the BAP. A separable $p$-Banach space is complementably universal for the BAP if it has the BAP and contains a complemented copy of each separable $p$-Banach space with the BAP. The existence of such spaces (one for each $0<p<1$) was first mentioned by Kalton himself in \cite[Theorem 4.1(b)]{k-sm}. A complete proof appears in the {\em related issues} of  \cite{ccm}. In any case, it easily follows from the Pe\l czy\'nski decomposition method that any two separable $p$-Banach spaces complementably universal for the BAP are isomorphic, so let us denote by $\mathscr K_p$ the isomorphic type of such specimens and observe that since each separable $p$-Banach space with the BAP is complemented in one with a basis, it follows that $\mathscr K_p$ does have a basis. Needless to say, $\mathscr K_p$ is not locally convex since it contains a complemented copy of $\ell_p$.

\begin{corollary}
If $K$ is a (non-empty) metrizable compactum of finite covering dimension, then $C(K,\mathscr K_p)$ is linearly homeomorphic to $\mathscr K_p$. In particular,  $C(I,\mathscr K_p)$ and $C(\Delta,\mathscr K_p)$  are linearly homeomorphic although $\mathscr K_p$ is not locally convex.
\end{corollary}

\begin{proof}
This clearly follows from the lemma since $C(K,\mathscr K_p)$ is separable, has the BAP and contains $\mathscr K_p$ complemented as the subspace of constant functions.
\end{proof}

We do not know of any other nonlocally convex quasi Banach space $X$ for which $C(I,X)$ and $C(\Delta, X)$ are isomorphic, apart from the obvious ones arising as direct sums of $\mathscr K_p$ and Banach spaces lacking the BAP. An obvious candidate is the $p${\em-Gurariy space}, introduced by Kalton in \cite[Theorem 4.3]{k-trans} and further studied in \cite{CGK}. Note that if $X$ is a quasi Banach space isomorphic to $X\oplus F$, with $F$ finite dimensional and $C(I,X)$ and $C(\Delta, X)$ are not isomorphic then neither are $C(I,X\oplus c_0)$ and $C(\Delta, X\oplus c_0)$.

It's time to leave.
Perhaps the most important question regarding the general topological properties of quasi Banach spaces is to know whether every quotient operator $Q:Z\To X$ (acting between quasi Banach spaces) admits a continuous section, namely  a continuous $\sigma: X\To Z$ such that $Q\circ \sigma={\bf I}_X$. More generally, let us say that $f\in C(K,X)$ lifts through $Q$ if there is $F\in C(K,Z)$ such that $f=F\circ Q$.
Now, given $0<p<1$, a quotient operator between $p$-Banach spaces $Q:Z\To X$ and a compactum $K$, consider the following statements:
\begin{enumerate}
\item $Q$ admits a continuous section.
\item Every continuous $f:K\To X$ has a lifting to $Z$.
\item $C(K)\otimes X$ is dense in $C(K,X)$.
\end{enumerate}
Clearly, (1)$\implies$(2): set $F=\sigma\circ f$, where $\sigma$ is the hypothesized section of $Q$.
Besides, if (1) is true for some quotient map $\ell_p(J)\To X$ then so it is for every $Q$. Similarly, if (2) is true for a given $K$ for some quotient map $\ell_p(J)\To X$, then it is true for any quotient map onto $X$ and (3) holds.

Following (badly) Klee \cite[Section~2]{klee}, let us say that the pair $(K,X)$ is {\em admissible} if (3) holds, that $K$ is admissible if (3) holds for every quasi Banach space $X$ and that $X$ is admissible if (3) holds for every compact $K$. We do not know whether the $p$-Gurariy spaces are admissible or not.

We have mentioned Shuchat's result that every compactum of finite covering dimension is admissible. Actually one can prove that (2) holds for any $Q$ if $\dim(K)<\infty$. This indeed follows from Michael's \cite[Theorem 1.2]{michael} but a simpler proof can be given using Shuchat's result, the argument of the proof of the lemma, and the open mapping theorem. Since every metrizable compactum is the continuous image of $\Delta$ this implies that for every compact subset $S\subset X$ there is a compact subset $T\subset Z$ such that $Q[T]=S$.

Long time ago, Riedrich proved that  the spaces $L_p$ are admissible for $0\leq p<1$; see \cite{i, capo} for more general results that cover all modular function spaces. We do not know if {\em the} quotient map $\ell_p\To L_p$ has a continuous section or satisfies (2) for arbitrary compact $K$ and $0<p<1$.


\begin{thebibliography}{99}


\bibitem{ccm} F. Cabello S\'anchez, J.M.F. Castillo, Y. Moreno, \emph{On the bounded approximation property on subspaces of $\ell_p$
when $0 < p < 1$ and related issues}, Forum Math. {\bf 14-08} (2019) 1--24.

\bibitem{CGK} F. Cabello S\'anchez, J. Garbuli\'nska-W\c egrzyn, W. Kubi\'s, \emph{
 Quasi-Banach spaces of almost universal disposition},  J. Funct. Anal. {\bf 267} (2014) 744--771.

\bibitem{capo}
D. Caponetti, G. Lewicki, {\em A note on the admissibility of modular function spaces}, J. Math. Anal. Appl. {\bf 448} (2017) 1331--1342.

\bibitem{cauty}
R. Cauty, {\em Un espace m\'etrique lin\'eaire qui n'est pas un r\'etracte absolu}, Fund. Math. {\bf 146} (1994) 85--99.


\bibitem{i}
J. Ishii, {\em On the admissibility of function spaces}, J. Fac. Sci., Hokkaido Univ., Ser. 1 {\bf 19} (1965) 49--55.


\bibitem{k-sm} N.J. Kalton, {\em Universal spaces and universal bases in metric linear
spaces}, Studia Math. {\bf 61} (1977) 161--191.

\bibitem{k-trans} N.J. Kalton, \emph{Transitivity and quotients of Orlicz spaces}, Comment.
Math. (Special issue in honor of the 75th birthday of W. Orlicz) (1978) 159--172.


\bibitem{k-pems}
N.J. Kalton, {\em Isomorphisms between spaces of vector-valued continuous functions}, Proc. Edinburgh Math. Soc. {\bf 26} (1983) 29--48.

\bibitem{d-k}
N.J. Kalton, T. Dobrowolski, {\em Cauty's space enhanced}, Topology Appl. {\bf 159} (2012) 28--33.

\bibitem{klee}
V. Klee, {\em Leray--Schauder theory without local convexity}, Math. Ann. {\bf 141} (1960) 286--296.


\bibitem{michael}
E. Michael, {\em Continuous selections II}, Ann. of Math. (2) {\bf 64} (1956) 562--580.

\bibitem{s-pams}
A.H. Shuchat, {\em Approximation of vector-valued continuous functions}, Proc. Amer. Math.
Soc. {\bf 31} (1972), 97--103.


\bibitem{w}
L. Waelbroeck,
{\em Topological vector spaces}, Summer school on
topological vector spaces, Bruxelles 1972, Springer Lecture
Notes in Math. {\bf 331}, Berlin-Heidelberg-New York, 1973, pp. 1--40.

\end{thebibliography}
\end{document}